\documentclass[12pt]{amsart}\usepackage{amsfonts,latexsym}
\newtheorem{theorem}{Theorem}
\newtheorem{lemma}[theorem]{Lemma}
\newtheorem{thm}[theorem]{Theorem}

\newtheorem{proposition}[theorem]{Proposition}

\theoremstyle{definition}
\newtheorem{definition}[theorem]{Definition}

\newtheorem{problems}[theorem]{Problems} 
\newtheorem{notations}[theorem]{Notations}

\theoremstyle{remark}
\newtheorem{remark}[theorem]{Remark}

\DeclareMathOperator{\cf}{cf}
\DeclareMathOperator{\CAP}{CAP}

\newcommand{\m}{\mathfrak}

\title[Combinatorial principles, compactness of spaces VI]
{Combinatorial and model-theoretical principles related to 
regularity of ultrafilters and compactness of topological spaces. VI.}

\author[]{Paolo Lipparini} 
\address{Dipartito di Matematica\\
Viale della Ricerca Scientifica\\
II Unidiversit\`a Roma (Tor Vergata)\\
I-00133 ROME 
ITALY}
\urladdr{http://www.mat.uniroma2.it/\textasciitilde lipparin}
\thanks{The author has received support from MPI and GNSAGA.
We wish to express our gratitude to X. Caicedo for stimulating discussions and correspondence} 


\begin{document} 

\begin{abstract} 
We discuss the existence of complete accumulation points of sequences 
in products of topological spaces. Then we collect and generalize
many of the results proved in Parts I, II and IV.

The present Part VI is complementary to Part V to the effect that
here we deal, say, with uniformity, complete accumulation points and 
$ \kappa $-$( \lambda )$-compactness,
rather than with regularity, $[ \lambda , \mu ]$-compactness and
$ \kappa $-$  ( \lambda , \mu  )$-compactness.
Of course, if we restrict ourselves to regular cardinals, Parts V 
(for $ \lambda = \mu$) and Part VI
essentially coincide.
\end{abstract}

\maketitle 

See Parts I - IV \cite{parti}, or 
\cite{BF,C,CK,CN,KM,KV,easter,jsl,bumi,arxiv,topproc,topappl,sa,VLNM}
for unexplained notation.

Let us recall the definition of the $\Box^ {<\nu} $ product.
If  $ \nu$ is a cardinal, and  
$(X_ \beta ) _{ \beta \in \kappa }$ is a family of topological spaces, 
then the
$\Box^ {<\nu} $ topology on the cartesian product
$\prod_{ \beta \in \kappa } X_ \beta  $
 is  the topology a base of which
is given by all products $ \prod _{ \beta \in \kappa }Y_ \beta $,
where each $Y_ \beta $ is an open subset of $X_ \beta $,
and $| \{  \beta \in \kappa | Y_ \beta \not= X_ \beta \}| < \nu$.   
The product of $(X_ \beta ) _{ \beta \in \kappa }$
with the $\Box^ {<\nu} $ topology shall be denoted by 
$\Box^{<\nu}_{ \beta \in \kappa } X_ \beta $.

Notice that in the case $\nu= \omega  $ we get the more usual
Tychonoff product. As usual, we shall denote the
Tychonoff product by $ \prod_{\beta\in\kappa } X_\beta    $.

Recall that, for every infinite cardinal $ \lambda $, a topological space 
$X$ is said to satisfy $\textrm{CAP}_\lambda $ if and only if 
every subset $Y\subseteq X$ with $|Y|=\lambda $ has a complete
accumulation point in $X$.

Actually, in what follows we shall use the variation of $\textrm{CAP}_\lambda $
which deals with sequences, rather than subsets.

If $ \lambda $ is an infinite cardinal, 
$X$ is a topological space 
and $(x_ \alpha )_{\alpha < \lambda } $
is a sequence of elements of $X$, 
we say that $x \in X$ is a $\lambda $-\emph{accumulation point}
for  $(x_ \alpha )_{\alpha < \lambda } $ if and only if 
$ |\{ \alpha < \lambda | x_ \alpha \in U \}|=  \lambda $,
for every neighborhood $U$ of $x$.

We say that a topological space 
$X$ satisfies $\textrm{CAP}^*_\lambda $ if and only if 
 every $ \lambda $-indexed sequence $(x_ \alpha )_{\alpha < \lambda } $
 of elements of $X$ has a $\lambda $-accumulation point.

\begin{proposition}\label{capcap}
(a) If $ \lambda $ is a regular cardinal, then $\textrm{CAP}^*_\lambda $ is equivalent to  $\textrm{CAP}_\lambda $, and equivalent
to $[\lambda,\lambda ]$-compactness.

(b) If $ \lambda $ is a singular cardinal, then $\textrm{CAP}^*_\lambda $ is equivalent to the conjunction of $\textrm{CAP}_\lambda $ and $\CAP_{\cf\lambda} $.
 \end{proposition}

Notice that the space $\cf \lambda $, endowed with the order topology, 
does not satisfy $\CAP_{\cf\lambda} $, but it satisfies
$\CAP_{ \kappa } $, for every infinite cardinal 
$ \kappa \not = \cf \lambda $.

On the other hand, if $\mu$ is regular, and  $\mu \leq \lambda  $, 
 consider the space
$X= S _{\mu}( \lambda )  $,
endowed with the topology which has, as a base,
all the sets of the form
$ \{ y \in S _{\mu  }( \lambda ) \mid y \subseteq  x  \} $,
$x$ varying in $S _{\mu }( \lambda )$.   
$X$ 
satisfies
$\CAP_{ \kappa } $, for every infinite cardinal 
$ \kappa < \mu $. Indeed, 
if $ \kappa < \mu $, then every
$ \kappa $-indexed sequence 
$(x_ \beta ) _{ \beta \in \kappa } $
in $X$ 
converges to  
$\bigcup_{ \beta \in \kappa } x_ \beta   $.
On the contrary, if $ \lambda \geq \kappa \geq \mu$, then
$X$ fails to satisfy  $\CAP_{ \kappa } $.

Thus, if in the above example we take
$ \lambda $ singular, and 
$\mu= (\cf \lambda )^+$,
we get a space satisfying 
$\CAP_{ \cf \lambda } $
but not $\CAP_{ \lambda } $.

If $ \lambda $ is singular, and 
$2^{\cf \lambda }< \lambda$,
then it is not difficult to construct a Tychonoff topological space
$X$ which does not satisfy
 $\CAP_{ \lambda } $, but such that
every power of $X$ 
  satisfies $\CAP_{ \cf \lambda } $.

Notice that if $ \lambda $ is singular,
$ \CAP _{\cf \lambda } $ holds,
and there are arbitrarily large $ \kappa < \lambda $
such that $ \CAP_ \kappa $ holds, then
$ \CAP_ \lambda $ holds, hence, by Proposition \ref{capcap}(b),
also $ \CAP^*_ \lambda $ holds.
The proof of the above fact is similar to the proof of the proposition
in \cite[p. 94]{sa}.   

\begin{proposition}\label{localdec}
Suppose that $X$ is a topological space, $x \in X$,
$ \lambda $ is an infinite cardinal  and
$(x_\gamma )_{ \gamma\in \lambda} $ is a sequence of elements of $X$.

Then  $x$ is a $\lambda $-accumulation point
for $(x_\gamma )_{ \gamma\in \lambda} $  
if and only if there exists some ultrafilter $D$
uniform over $\lambda $ such that
$(x_\gamma )_{ \gamma\in \lambda} $ $D$-converges to x.
 \end{proposition}

\begin{remark}\label{rmkreg}
Variations on the above proposition are well known. See e.g.
\cite[p. 80-81]{sa}.

An analogue of the above proposition for 
$(\lambda', \lambda)$-regular ultrafilters 
and  $[\lambda, \lambda']$-compactness is proved in \cite{Cprepr,C}.

For $\lambda=\lambda '$ a regular cardinal,
the statement in 
Proposition  \ref{localdec} and the corresponding result
from \cite{Cprepr,C} essentially amount to the same result, since 
an ultrafilter is $(\lambda, \lambda )$-regular 
if and only if it has some quotient uniform over
$\lambda $, and since $\textrm{CAP}^*_\lambda $ is equivalent to  
$[\lambda,\lambda ]$-compactness.
 \end{remark} 

\begin{notations}\label{notseq}
Since we will often be working in products, dealing with sequences indexed by 
external sets, in order to avoid confusion
we shall introduce a special notation to denote the
elements of the product.

If $ x \in \prod_{ \beta \in \kappa } X_ \beta  $,
say $x=(x_\beta )_{\beta \in \kappa }$,
we shall sometimes denote $x$ by
$\prod_{\beta\in\kappa } x_\beta $.
 \end{notations}

The next lemma generalizes the fundamental property 
of $D$-convergence with respect to products.
Indeed, since every ultrafilter is $\omega $-complete,
the particular case $\nu=\omega $ of the next lemma asserts that
$D$-convergence in a Tychonoff product is equivalent to 
factor by factor $D$-convergence.
The lemma is more general in the sense that it asserts the above equivalence
for $\nu$-complete ultrafilters and 
$\Box^{<\nu}$ products.

\begin{lemma}\label{convprod}
Suppose that $D$ is a $\nu$-complete ultrafilter over some set $I$,
$(X_\beta )_{\beta\in\kappa }$ is a family of topological spaces,
and $(x_i)_{i \in I}$ is a sequence of elements of
$X=\Box^{<\nu}_{ \beta \in \kappa } X_ \beta $,
say, for each $i \in I$,
$x_i=\prod_{\beta \in \kappa} x_{i,\beta }$.

Then $(x_i)_{i \in I}$ $D$-converges in $X$ 
to some $x = \prod_{\beta \in \kappa} x_\beta \in X$ if and only if,
for each $\beta\in\kappa $,
$(x_{i,\beta })_{i \in I}$ $D$-converges to $x_\beta $ in $X_\beta $.
 \end{lemma}

\begin{proposition}\label{lmprop}
Suppose that  $\kappa$ is a cardinal, 
and  $ \lambda $, $\nu$,
$(\mu_\beta )_{\beta\in\kappa }$ are infinite
cardinals. For each $\beta \in \kappa $,
endow $\mu_\beta $ with the order topology.
Suppose that $ (f_ \beta ) _{ \beta \in \kappa } $
is a given set of functions such that 
$f_\beta:\lambda\to \mu_\beta $,
for each $\beta \in \kappa $.

Consider the following conditions.

(a) There exists a $\nu$-complete
 ultrafilter  $D$ uniform over $\lambda $
such that for no $\beta\in \kappa $
$f_\beta (D)$ is uniform over $\mu_\beta $.

(b) In the space 
$\Box^{<\nu}_{ \beta \in \kappa } \mu_ \beta $
the sequence $(x_\gamma  )_{\gamma<\lambda } $
defined by $x_\gamma = \prod_{\beta \in \kappa}f_\beta (\gamma )$
has a $\lambda $-accumulation point
in $\Box^{<\nu}_{ \beta \in \kappa } \mu_ \beta $.
 
If for each $\beta\in\kappa $ $\mu_\beta $
is a regular cardinal, then Condition (a) implies Condition (b). 

If $\nu=\omega $ then Condition (b) implies Condition (a).
\end{proposition}

\begin{proof}
(a)$\Rightarrow $(b). Fix $\beta\in \kappa $.
Since $f_\beta (D)$ is not uniform over $\mu_\beta $, and
$\mu_\beta $ is a regular cardinal, then $f_\beta (D)$
is not $(\mu_\beta, \mu_\beta)$-regular.
By \cite[Proposition 1]{topproc}, $\mu_\beta $
is $f_\beta (D)$-compact, that is, every $f_\beta (D)$-sequence
in $\mu_\beta $ $f_\beta (D)$-converges.
In particular, the identity function on $f_\beta (\lambda )$
$f_\beta (D)$-converges to some point of $\mu_\beta $.
By definition of $f_\beta (D)$ and of ultrafilter convergence, 
this implies that the sequence
$ (f_\beta (\gamma ))_{\gamma< \lambda }$
$D$-converges in $\mu_\beta $
to some point, say, to $ x_ \beta $.

Since the above holds for every $\beta \in \kappa $,
and 
$(x_\gamma)_\beta  = f_\beta (\gamma )$, then,
by Lemma \ref{convprod}, 
the sequence $(x_\gamma  )_{\gamma<\lambda } $
 $D$-converges in $\Box^{<\nu}_{ \beta \in \kappa } \mu_ \beta $ to  
$x = \prod _{ \beta \in \kappa } x_ \beta  $.
By Proposition \ref{localdec}, $x$ is a
$\lambda $-accumulation point for $(x_\gamma  )_{\gamma<\lambda } $.

Now suppose that 
$\nu= \omega $. We are going to show that
(b) implies (a). Suppose
that (b) holds.
By Proposition 
\ref{localdec}
 there exists some ultrafilter $D$
uniform over $\lambda $ such that
 $(x_ \gamma)_{\gamma  < \lambda } $ $D$-converges 
to some point of $\prod_{ \beta \in \kappa } \mu_ \beta $.

By Lemma \ref{convprod}, for each $\beta\in \kappa $,
the sequence $((x_\gamma)_\beta)_{\gamma<\lambda }$
$D$-converges
to some point of $\mu_\beta $.
Since $ (x_\gamma)_\beta = f_\beta (\gamma )$, this implies
 that, for each $\beta\in \kappa $, the identity function on 
$ \mu_\beta $  $f_\beta (D)$-converges
to some point of $\mu_\beta $, and this easily implies that 
$f_\beta (D)$ is not uniform over $\mu_\beta $. 
\end{proof}

\begin{definition} \label{models}
Suppose that  $ \lambda $ is an infinite
 cardinal, and $ \m M$ is a model with a unary
predicate $U$ and a distinguished binary relation $<$ such that
$ \langle U^{\m M}, <^{\m M}\rangle = \langle \lambda , <\rangle $.
Suppose further that ${\m M}$ has
a name for every element of $U= \lambda $
(for simplicity, and by abuse of notation, we shall suppose that the name for 
$ \alpha $ is $ \alpha $ itself). 

If $ {\m N} \equiv {\m M} $, and $ b \in N$, we say that  
$ b $ is $ \lambda $-\emph{non standard}    
if and only if, in ${\m N}$, $U(b)$ holds,  and  $ \alpha < b$
holds for every $ \alpha < \lambda $.  
Other sentences which have been used to indicate similar notions are that
``$b$ bounds $ \lambda $'', or that 
``$b$ realizes $ \lambda $''.
Of course, in the particular case  $ \lambda = \omega $,
we get the usual notion of a non-standard element. 
In other words, 
${\m N}$ has a 
$ \lambda $-non standard element
if and only if 
$ \langle U ^{ {\m N} } , <^{\m N}\rangle$
is not an \emph{end extension} of 
$ \langle \lambda  , <\rangle$.

 We shall use the above terminology even when 
$U$ is not a predicate in the vocabulary of ${\m M}$,
but  just a unary relation definable by some formula.
In particular, if $\mu< \lambda $, and, as above, in ${\m M}$,
$ \langle U, <\rangle = \langle \lambda , <\rangle $,
 ${\m M}$ has
a name for every element of $ \lambda $, and $ {\m N} \equiv {\m M} $, we shall say that
$ b $ is $ \mu $-non standard in ${\m N}$    
if and only if, in ${\m N}$, $ b < \mu$ holds,  and  $ \alpha < b$
holds for every $ \alpha < \mu $.

Of course, it might be the case that the model ${\m M}$ 
has many predicates $U_i$  and relations $R_i$ such that
$ \langle U_i, R_i\rangle \cong \langle \lambda, <\rangle $.
If this is the case, and ${\m N} \equiv {\m M} $,
it might happen that ${\m N}$ has a $ \lambda $-non standard   
element according to, say, $U_0$, $R_0$, but no 
$ \lambda $-non standard   
element according to  $U_1$, $R_1$.
We shall try to adhere to the convention that
the base set of ${\m M}$ does always  contain the ``real''
$ \lambda $, so that our definition of 
$ \lambda $-non standard is not ambiguous.
Anyway, the above possible ambiguity is not a
serious problem, as far as the present paper is concerned,
since we always allow ${\m M}$ to be expanded with additional
functions and relations, so that we can always have, inside ${\m M}$,
isomorphisms between any $ \langle U_i, R_i\rangle$
and any $ \langle U_j, R_j\rangle$. 
Since being an isomorphism is expressible by first order 
sentences, and ${\m N}$  is supposed to be elementarily
equivalent to ${\m M}$, then a $ \lambda $-non standard element
exists in ${\m N}$ according to    $U_i $, $  R_i$ if and only if
a $ \lambda $-non standard element exists according to $U_j $, $  R_j$.
  \end{definition}  

As far as Condition (7) below is concerned, 
fix some set $V \subseteq S_ \lambda  ( \lambda )$
 cofinal in  $ S_ \lambda  ( \lambda )$
of cardinality $\leq \kappa $ and, for  $ v \in V$,
 let $ R_v$ be the unary predicate on 
$ \lambda $ defined by $  R_v( \gamma )$
if and only if $ \gamma \in v$
 (compare Part IV, Definition 2 and Remark 3). 

As far as Condition (8) below is concerned,
recall the definition of a $ \kappa $-$  ( \lambda )$-compact
 logic from Part IV, Definition 10.

\begin{theorem}\label{lmtopol} 
Suppose that  $\kappa$ is a cardinal, 
 $ \lambda $ is an infinite cardinal, and 
$(\mu_\beta )_{\beta\in\kappa }$ is a set of infinite
regular cardinals.

Then the following conditions are equivalent.
\begin{enumerate}

\item 
There are $\kappa $ functions $ (f_ \beta ) _{ \beta \in \kappa } $ 
such that 
\begin{enumerate}
\item
For each $\beta \in \kappa $,
$f_\beta:\lambda\to \mu_\beta $;
 and  
\item
for every  ultrafilter  $D$ uniform over $\lambda $
there is $\beta\in \kappa $ such that 
$f_\beta (D)$ is uniform over $\mu_\beta $.
\end{enumerate}

\item 
 There are $ \kappa $ functions $ (f_ \beta ) _{ \beta \in \kappa } $
such that 
\begin{enumerate}
\item
For each $\beta \in \kappa $,
$f_\beta:\lambda\to \mu_\beta $;
and
\item 
for every function $g \in \prod_{\beta\in \kappa } \mu_\beta  $
 there exists some finite
set $F \subseteq \kappa $ such that 
$ \left| \bigcap _{\beta \in F} f_\beta  ^{-1}([0, g(\beta ))) \right| < \lambda $.
\end{enumerate}

\item 
There is a family $ (B_{ \alpha , \beta }) _{ \beta \in \kappa, \alpha<\mu_\beta  }  $ 
of subsets of $ \lambda $ such that:
\begin{enumerate}
\item 
For every $ \beta \in \kappa$, $\bigcup _{ \alpha<\mu_\beta  } B_{ \alpha , \beta  } = \lambda$;
\item
For every $ \beta \in \kappa$ and $ \alpha \leq \alpha ' < \mu_\beta   $, 
$ B_{ \alpha , \beta } \subseteq B_{ \alpha' , \beta }$;
\item
For every function $ g \in \prod_{\beta\in \kappa } \mu_\beta  $ there exists a finite set
$F \subseteq \kappa  $ such that 
$|\bigcap _{\beta \in F} B_{ g( \beta) , \beta }| < \lambda $.   
\end{enumerate} 

\item 
For every family 
$(X_ \beta ) _{ \beta \in \kappa }$ of topological spaces,
if 
$\prod_{ \beta \in \kappa } X_ \beta $
satisfies $\CAP^*_\lambda $,
then there is $\beta  \in \kappa $
such that  $X_ \beta $ 
satisfies $\CAP^* _{\mu_ \beta }$.

\item 
The topological space $ \prod_{\beta\in\kappa }\mu_\beta $ does not
satisfies $\CAP^*_\lambda $, where each $ \mu_\beta $
is endowed with the  topology whose open sets are the intervals
$[0, \alpha )$ ($ \alpha \leq \mu_\beta $).

\item 
The topological space $ \prod_{\beta\in\kappa }\mu_\beta $ does not
satisfies $\CAP^*_\lambda $, where each $ \mu_\beta $
is endowed with the order topology.
\end{enumerate} 

If $\lambda \geq \mu_\beta $, for every $ \beta \in \kappa $, and $\kappa \geq \cf S_ \lambda  ( \lambda )$,
then the preceding conditions are also equivalent to the following ones:
\begin{enumerate} 

\item[(7)]
The model 
$  \langle  \lambda, <, R_v, \gamma  \rangle _{v \in V , \gamma < \lambda}   $
 has an expansion (equivalently, a multi-sorted expansion) ${\m A}$ in a language with at most $ \kappa $ new symbols 
(and  sorts) such that whenever
$\m B \equiv \m A$ and
$ \m B$ 
(respectively, $ \lambda ^{\m B} $) has 
 an element $x$ such that 
$ {\m B} \models \neg R_v( x) $ for every $ v \in V$,
then there exists $ \beta \in \kappa $ such that  
$ \m B$ has a $ \mu_ \beta $-non standard element. 

\item[(8)] 
Every  
$ \kappa $-$( \lambda )$-compact
 logic (equivalently,
every  
$ \kappa $-$( \lambda )$-compact
 logic
 generated by  
$ \sup _{ \beta \in \kappa } \mu_ \beta   $ 
cardinality quantifiers)
is 
$ \kappa $-$  ( \mu_ \beta , \mu_ \beta )$-compact
for some $ \beta \in \kappa $.
\end{enumerate} 

If in addition $\lambda $ is a regular cardinal,
then the preceding conditions are also equivalent to the following one:

\begin{enumerate} 
\item[(9)] The model $ \langle \lambda, <, \gamma \rangle _{ \gamma < \lambda }  $ has an
expansion (equivalently, a multi-sorted expansion) ${\m A}$ in a language with at most $ \kappa $ new symbols (and sorts) such that whenever
$\m B \equiv \m A$ and $ \m B$ 
(respectively, $ \lambda ^{\m B} $) has a
$ \lambda $-non standard element, then, for
some $ \beta \in \kappa $,  
$ \m B$ has a
$ \mu _ \beta  $-non standard element.
\end{enumerate}
\end{theorem}

\begin{proof}
The equivalence of Conditions (1)-(3) is proved as in Part II, Theorem 1, 
equivalence of Conditions (b), (b$'$), (c).

(1)$\Rightarrow $(4). Suppose by contradiction that (1) holds and (4) fails.
Thus, there are topological spaces
$(X_ \beta ) _{ \beta \in \kappa }$ such that 
$X=\prod_{ \beta \in \kappa } X_ \beta $
satisfies $\CAP^*_\lambda $,
but, for every  $\beta  \in \kappa $,
 $X_ \beta $ fails to satisfy
 $\CAP^* _{\mu_\beta} $.

This means that for every $ \beta \in \kappa $ 
there exists a sequence $((y _\beta )_ \alpha ) _{ \alpha < \mu_ \beta } $
which has no $ \mu_ \beta $-accumulation
point in $ X_ \beta $.

Suppose that 
$ (f_ \beta ) _{ \beta \in \kappa } $ are functions as given by (1).
Define a $ \lambda $-indexed sequence
$ (x_ \gamma ) _{ \gamma < \lambda }  $ in
$X=\prod_{ \beta \in \kappa } X_ \beta $, 
as follows.
For $\gamma < \lambda $, 
$x_ \gamma = \prod _{ \beta \in \kappa } x _{ \gamma, \beta }   $
with $x _{ \gamma, \beta }= (y _{ \beta }) _{f _ \beta ( \gamma )}  $.

Since $X$ satisfies 
$\CAP^*_\lambda $,
the sequence 
$ (x_ \gamma ) _{ \gamma < \lambda }  $
has a $ \lambda $-accumulation point $x \in X$.
By Proposition \ref{localdec}, 
there exists an ultrafilter $D$ uniform over
$ \lambda $ such that  
$ (x_ \gamma ) _{ \gamma < \lambda }  $
$D$-converges to $x$.

By Lemma \ref{convprod}, 
for every $ \beta \in \kappa $, 
the sequence $(x _{ \gamma, \beta }) _{ \gamma < \lambda } $
$D$-converges in $X_ \beta $.
Since $x _{ \gamma, \beta }= (y _{ \beta }) _{f _ \beta ( \gamma )}  $,
we have that,  
for every $ \beta \in \kappa $, 
the sequence $((y _\beta )_ \alpha ) _{ \alpha < \mu_ \beta } $
$f _ \beta (D)$-converges to some point of  
$ X_ \beta $. 

By (1), there exists some $ \bar \beta \in \kappa $
such that  $f _ {\bar \beta} (D)$ is uniform over $ \mu _{\bar \beta} $, but then 
Proposition  \ref{localdec} implies that 
$((y _{\bar{\beta}} )_ \alpha ) _{ \alpha < \mu_ {\bar{\beta}}  } $
has a $ \mu _ {\bar{\beta} } $-accumulation point, contradiction.

(4)$\Rightarrow $(5) is trivial, since $\mu_ \beta$ does not satisfies 
$\CAP^* _{\mu_ \beta }   $, recalling the assumption that each
$\mu_ \beta$ is a regular cardinal.

(5)$\Rightarrow $(6) is trivial, since the topology in (6) is
finer than the topology in (5).

(6)$\Rightarrow $(1). If (6) holds, then there exists a $ \lambda $-indexed
sequence 
$(x_\gamma  )_{\gamma<\lambda } $
in $X$  
having no $\lambda $-accumulation point in $X$.
Say, $x_\gamma = \prod_{\beta \in \kappa} (x_ \gamma )_ \beta $.

For $ \beta \in \kappa $ and  $ \gamma < \lambda $, 
define $f_\beta (\gamma )=(x_ \gamma )_ \beta$.
The contrapositive of Proposition \ref{lmprop} (a)$\Rightarrow $(b)
then implies (1). 

The equivalence of (1),
(7) and (9) is proved as in Part IV, Theorem 7 
(Cf. also Part V, Theorem 1.2, Conditions (a), (d), (e)).

The equivalence of Condition (8) with the other conditions shall be presented elsewhere.
\end{proof}

\begin{proposition}\label{nuoprop}
Suppose that Condition (4) in Theorem \ref{lmtopol} holds. 
Then the following holds.
 
If
$(Y_ j ) _{ j \in J }$ is a family of topological spaces,
$\prod_{ j \in J } Y_ j $
satisfies $\CAP^*_\lambda $ and
if, for every  $\beta  \in \kappa $,
we put $J_ \beta = \{j \in J \mid Y_j \text{ does not  satisfy } \CAP^* _{\mu_ \beta } \} $,
then there is $\beta  \in \kappa $ such that 
$ | J_ \beta | \leq| \beta | $. 
 \end{proposition}

\begin{proof}
Suppose that Condition (4) in Theorem \ref{lmtopol} holds,
 but the conclusion of Proposition \ref{nuoprop} fails.
Thus, there exists a product
$Y=\prod_{ j \in J } Y_ j $
satisfying $\CAP^*_\lambda $, but
such that $ |  J_ \beta | > |\beta |  $,
for every $\beta  \in \kappa $.
Hence we can inductively construct a sequence
$(j_ \beta ) _{ \beta  < \kappa } $ 
such that, 
for every $ \beta < \kappa $,
$Y _{j_ \beta } $ fails to satisfy  
$ \CAP^* _{\mu_ \beta }$ and, moreover, for
$ \beta < \beta' < \kappa $, 
$j_ \beta \not = j _{ \beta' } $.

For $\beta \in \kappa $, put $X_\beta =Y_{j_\beta }$. 
Hence, $\prod_{ \beta \in \kappa } X_{\beta }$
satisfies $\CAP^*_\lambda $, since we are taking the product of 
a subset of the factors of $Y$. But then 
(4) implies that some $X_\beta = Y_{j _{ \beta }}$
satisfies  $ \CAP^* _{\mu_ \beta }$, a contradiction.
\end{proof}

We have a partial version of Theorem \ref{lmtopol} for the case when 
the $ \mu_ \beta $'s are not necessarily  regular.

As far as Conditions (4)(5) below are concerned, 
for each $ \beta \in \kappa $, choose some set 
$V_ \beta  \subseteq S_{\mu_\beta } (\mu_\beta  )$
 cofinal in  $S_{\mu_\beta } (\mu_\beta  )$
of cardinality $\leq \kappa $ and, for  $ v \in V_ \beta $,
 let $ R_v$ be the unary predicate on 
$ \mu_\beta  $ defined by $  R_v( \alpha  )$
if and only if $ \alpha \in v$.

 \begin{thm}\label{lmkprod4gen} 
Suppose that  $\kappa$ is a cardinal, 
 $ \lambda $ is an infinite cardinal, and 
$(\mu_\beta )_{\beta\in\kappa }$ is a set of infinite
cardinals.
Then the following conditions are equivalent.
\begin{enumerate}   
\item 
There are $\kappa $ functions $ (f_ \beta ) _{ \beta \in \kappa } $ 
such that 
\begin{enumerate}
\item
For each $\beta \in \kappa $,
$f_\beta:\lambda\to \mu_\beta $;
 and  
\item
for every  ultrafilter  $D$ uniform over $\lambda $
there is $\beta\in \kappa $ such that 
$f_\beta (D)$ is uniform over $\mu_\beta $.
\end{enumerate}

\item 
 There are $ \kappa $ functions $ (f_ \beta ) _{ \beta \in \kappa } $
such that 
\begin{enumerate}
\item
For each $\beta \in \kappa $,
$f_\beta:\lambda\to \mu_\beta $;
and
\item 
for every function $g: \in \prod _{ \beta \in \kappa }  S _{\mu_ \beta } (\mu_ \beta )$
 there exists some finite
set $F \subseteq \kappa $ such that 
$ \left| \bigcap _{\beta \in F} f_\beta  ^{-1}(g(\beta )) \right| < \lambda $.
\end{enumerate}

\item
There is a family $ (C_{ \alpha , \beta }) _{  \beta \in \kappa, \alpha \in \mu _ \beta }  $ 
of subsets of $ \lambda $ such that: 

\begin{enumerate}   
\item 
For every $ \beta \in \kappa$, 
$ (C_{ \alpha , \beta }) _{ \alpha \in \mu _ \beta }  $
is a partition of $ \lambda $.
\item
For every function $g \in \prod _{ \beta \in \kappa } S _{\mu _ \beta } (\mu_ \beta ) $ there exists a finite subset
$F \subseteq \kappa  $ such that 
$|\bigcap _{\beta \in F} \bigcup _{ \alpha \in g( \beta )} C_{ \alpha  , \beta }| < \lambda $.
\end{enumerate}

\end{enumerate}   

If $\kappa \geq \cf S_ \lambda  ( \lambda )$, 
and $\lambda \geq \mu_\beta $, 
$ \kappa \geq \cf S _{ \mu_ \beta }(\mu _ \beta ) $, for every $ \beta \in \kappa $,  
then the preceding conditions are also equivalent to the following ones:

\begin{enumerate} 
\item[(4)]
The model 
$ \langle  \lambda, <, R_v, \gamma  \rangle _{v \in V \cup  V_\beta, \gamma < \lambda }   $
 has an expansion (equivalently, a multi-sorted expansion) ${\m A}$ in a language with at most $ \kappa $ new symbols 
(and sorts) such that whenever
$\m B \equiv \m A$ and  $ \lambda ^ {\m B} $ 
has an element $x$ such that 
$ {\m B} \models \neg R_v( x) $ for every $ v \in V$,
then there exists $ \beta \in \kappa $ such that  
$ \m B$ has a 
has an element $y$ in $ \mu _\beta ^ {\m B} $ such that 
$ {\m B} \models \neg R_v( y) $ for every $ v \in V_\beta $.

\end{enumerate} 

If in addition $\lambda $ is a regular cardinal,
then the preceding conditions are also equivalent to the following one:

\begin{enumerate} 
\item[(5)] 

The model 
$ \langle  \lambda, <, R_v, \gamma  \rangle _{v \in V_\beta, \gamma < \lambda }   $
 has an expansion (equivalently, a multi-sorted expansion) ${\m A}$ in a language with at most $ \kappa $ new symbols 
(and  sorts) such that whenever
$\m B \equiv \m A$ and  ${\m B} $ 
 has a
$ \lambda $-non standard element, 
then there exists $ \beta \in \kappa $ such that  
$ \m B$ has a 
has an element $y$ in $ \mu _\beta ^ {\m B} $ such that 
$ {\m B} \models \neg R_v( y) $ for every $ v \in V_\beta $.
\end{enumerate}
\end{thm} 

\begin{remark} \label{ultrem}
Of course, there is the possibility of proving a mix between
Theorems \ref{lmtopol} and \ref{lmkprod4gen}, in the case when
certain $\mu_ \beta $'s are regular and
other $\mu_ \beta $'s are singular.

We leave details to the reader.
 \end{remark}

\begin{problems}\label{probfin}
(a) We do not know whether we can extend Theorem \ref{lmkprod4gen},
for the case when the $ \mu _ \beta $'s are singular,
by adding further equivalent conditions analogue to 
Conditions (4)-(6) and (8) in Theorem \ref{lmtopol}.

(b) Does $ \lambda \stackrel{ \lambda^+ }{\Rightarrow} \mu  $
implies 
 $ \lambda \stackrel{ 2^\lambda }{\Rightarrow} \mu  $?

(c) Find conditions equivalent to the conditions in Theorems
\ref{lmtopol} and \ref{lmkprod4gen} which are expressed in terms of Boolean Algebras. 
\end{problems}

\begin{remark}\label{primeid} 
In most cases, in our proofs, we are not using the full axiom of choice, but only the Prime Ideal Theorem. 
\end{remark}

\end{document}